\documentclass[12pt,dvips]{amsart}
\usepackage{amsfonts, amssymb, latexsym, epsfig, amscd}

\newtheorem{Theorem}{Theorem}[section]
\newtheorem{Proposition}[Theorem]{Proposition}
\newtheorem{Lemma}[Theorem]{Lemma}

\newtheorem{Definition-Proposition}[Theorem]{Definition-Theorem}
\newtheorem{Main Conjecture}[Theorem]{Main Conjecture}

\newtheorem{Conjecture}[Theorem]{Conjecture}

\theoremstyle{remark}
\newtheorem{Example}[Theorem]{Example}

\newcommand{\excise}[1]{}

\newcommand\Groth{{\mathfrak G}}

\theoremstyle{plain}

\newcommand{\cellsize}{12}
\newlength{\cellsz} 
\newsavebox{\cell}
\sbox{\cell}{\begin{picture}(\cellsize,\cellsize)
\put(0,0){\line(1,0){\cellsize}}
\put(0,0){\line(0,1){\cellsize}}
\put(\cellsize,0){\line(0,1){\cellsize}}
\put(0,\cellsize){\line(1,0){\cellsize}}
\end{picture}}
\newcommand\cellify[1]{\def\thearg{#1}\def\nothing{}%
\ifx\thearg\nothing
\vrule width0pt height\cellsz depth0pt\else
\hbox to 0pt{\usebox{\cell} \hss}\fi%
\vbox to \cellsz{
\vss
\hbox to \cellsz{\hss$#1$\hss}
\vss}}
\newcommand\tableau[1]{\vtop{\let\\\cr
\baselineskip -16000pt \lineskiplimit 16000pt \lineskip 0pt
\ialign{&\cellify{##}\cr#1\crcr}}}
\newcommand{\kellsize}{10}
\newlength{\kellsz} 
\newsavebox{\kell}
\sbox{\kell}{\begin{picture}(\kellsize,\kellsize)
\put(0,0){\line(1,0){\kellsize}}
\put(0,0){\line(0,1){\kellsize}}
\put(\kellsize,0){\line(0,1){\kellsize}}
\put(0,\kellsize){\line(1,0){\kellsize}}
\end{picture}}
\newcommand\kellify[1]{\def\thearg{#1}\def\nothing{}%
\ifx\thearg\nothing
\vrule width0pt height\kellsz depth0pt\else
\hbox to 0pt{\usebox{\kell} \hss}\fi%
\vbox to \kellsz{
\vss
\hbox to \kellsz{\hss$#1$\hss}
\vss}}
\newcommand\ktableau[1]{\vtop{\let\\\cr
\baselineskip -16000pt \lineskiplimit 16000pt \lineskip 0pt
\ialign{&\kellify{##}\cr#1\crcr}}}

\begin{document}
\pagestyle{plain}
\title{Combinatorial rules for three bases of polynomials}

\author{Colleen Ross}
\address{Department of Industrial Engineering and Management Sciences\\
Northwestern University\\ Evanston, IL 60208}
\email{ColleenRoss2012@u.northwestern.edu}
\author{Alexander Yong}
\address{Department of Mathematics\\
University of Illinois at Urbana-Champaign\\
Urbana, IL 61801}
\email{ayong@illinois.edu}



\date{February 1, 2013}
\maketitle

\section{Introduction}

We present combinatorial rules (one theorem and two conjectures)
concerning three bases of ${\sf Pol}={\mathbb Z}[x_1,x_2,\ldots]$. 

Consider a basic question (studied for example in \cite{Poly}):
\begin{quotation}
How does one
 lift properties of the ring $\Lambda$ of symmetric functions (and its Schur basis)
to the entirety of ${\sf Pol}$?
\end{quotation}
The bases below lift the Schur polynomials. However, one wishes to analogize
the relationship in $\Lambda$ between rules for Schur polynomials and  
Littlewood-Richardson rules. For these bases, no 
rule has yet provided a parallel, explaining a desire for alternative forms.

First, we prove a ``splitting'' rule for the basis of
\emph{key polynomials} $\{\kappa_{\alpha} | \alpha\in {\mathbb Z}_{\geq 0}^{\infty}\}$, thereby establishing a new positivity theorem about them. This family was introduced by \cite{Demazure} and first studied combinatorially in \cite{LS2,LS1}. Combinatorial rules for their monomial
expansion are known, see, e.g., \cite{LS2, LS1, Reiner.Shimozono, HHL}. Our rule refines 
\cite[Theorem 5(1)]{Reiner.Shimozono} and is compatible with the splitting rule \cite[Corollary~3]{BKTY} for
the basis of \emph{Schubert polynomials} $\{{\mathfrak S}_w |w\in S_{\infty}\}$.

Second, we investigate a basis
$\{\Omega_{\alpha}|\alpha\in {\mathbb Z}_{\geq 0}^{\infty}\}$  
defined by \cite{Lascoux:trans}
that deforms the key basis. By extending the \emph{Kohnert moves} of \cite{Kohnert} we conjecturally
give the first combinatorial rule for the $\Omega$-polynomials.

Third, in \cite{Kohnert}, the Kohnert moves were used to conjecture the first combinatorial
rule for Schubert polynomials (a proof was later presented in \cite{Winkel}). Similarly, we use the extended Kohnert moves
to give a conjecture for the basis of \emph{Grothendieck polynomials} $\{{\mathfrak G}_w |w\in S_{\infty}\}$ \cite{LasSch2}. 
This rule appears significantly different than earlier (proved) rules, such as those in \cite{FK, Lascoux:trans, BKTY:II, LRS}. 

\subsection{Splitting key polynomials}

Let $S_{\infty}$ be the group of permutations of ${\mathbb N}$ with finitely many non-fixed points.
This acts on ${\sf Pol}$ by permuting the variables.
Let $s_i$ be the simple transposition interchanging $x_i$ and $x_{i+1}$. The {\bf divided difference operator} acts on ${\sf Pol}$ by
\[\partial_i=\frac{1-s_i}{x_i-x_{i+1}}.\]
Define the {\bf Demazure operator} by setting
\[\pi_i(f)=\partial_i(x_i \cdot f ), \mbox{ \ for $f\in {\sf Pol}$.}\]

For $\alpha=(\alpha_1,\alpha_2,\ldots)\in {\mathbb Z}_{\geq 0}^\infty$, the {\bf key polynomial} $\kappa_{\alpha}$ is
\[\kappa_{\alpha}=x_1^{\alpha_1}x_2^{\alpha_2}\cdots, \  \mbox{\ if $\alpha$ is weakly decreasing.}\]
Otherwise,
\[\kappa_{\alpha}=\pi_i(\kappa_{\widehat \alpha}) \mbox{\ where $\widehat\alpha=(\ldots,\alpha_{i+1},\alpha_i,\ldots)$ and
$\alpha_{i+1}>\alpha_{i}$.}\]
Since the leading term of $\kappa_{\alpha}$ is $x_1^{\alpha_1}x_2^{\alpha_2}\cdots$, the key polynomials
form a ${\mathbb Z}$-basis of ${\sf Pol}$.

The key polynomials lift the Schur polynomials: when
\begin{equation}
\label{eqn:kappaequalsschurcond}
\alpha=(\alpha_1,\alpha_2,\ldots,\alpha_t,0,0,0,\ldots), \mbox{\ where $\alpha_1\leq \alpha_2\leq\ldots\leq\alpha_t$, then}
\end{equation}
\begin{equation}
\label{eqn:kappaequalsschur}
\kappa_{\alpha}=s_{(\alpha_t,\cdots,\alpha_2,\alpha_1)}(x_1,\ldots,x_t).
\end{equation}

A {\bf descent} of $\alpha$ is an index $i$ such that $\alpha_{i}\geq 
\alpha_{i+1}$; a {\bf strict descent} is an index
$i$ such that $\alpha_{i}>\alpha_{i+1}$. Fix descents 
$d_1<d_2<\ldots<d_k$ of $\alpha$ containing all 
strict descents of $\alpha$. Since $\pi_i$ symmetrizes $\{x_i,x_{i+1}\}$, $\kappa_{\alpha}$ is separately symmetric in each collection: 
\[X_1 = \{x_1,x_2,\ldots,x_{d_1}\},
X_2 =  \{x_{d_1 +1}, x_{d_1 +2},\ldots,x_{d_2}\}, \ldots, 
X_k =  \{x_{d_{k-1}+1},x_{d_{k-1}+2},\ldots,x_{d_k}\}.\]
(The variables $x_{d_k+1},x_{d_{k}+2},\cdots$ do not appear in $\kappa_{\alpha}$.)
Therefore, uniquely:
\begin{equation}
\label{eqn:theexpansion}
\kappa_{\alpha}(X)=\sum_{\lambda^1,\ldots,\lambda^k} {\mathcal E}^{\alpha}_{\lambda^1,\ldots,\lambda^k} \ s_{\lambda^1}(X_1)\cdots
s_{\lambda^k}(X_k),
\end{equation}
where each $\lambda^i$ is a partition. \emph{A priori} one only knows ${\mathcal E}^{\alpha}_{\lambda^1,\ldots,\lambda^k}\in {\mathbb Z}$.

Given $\alpha\in {\mathbb Z}_{\geq 0}^{\infty}$, there is a unique
$w[\alpha]\in S_{\infty}$ such that ${\tt code}(w[\alpha])=\alpha$ 
(see, e.g., \cite[Proposition~2.1.2]{Manivel}). Here 
${\tt code}(w[\alpha])\in {\mathbb Z}^{\infty}_{\geq 0}$ counts the number of boxes in 
columns of ${\tt Rothe}(w[\alpha])$. We will need a special tableau
coming from \cite[Section~4]{Stanley}:

\noindent
{\sf The tableau} $T[\alpha]$: Given $w[\alpha]$, $i_1<i_2<\ldots<i_a$
in the first column of $T[\alpha]$ are given by having $i_j$ be the largest descent position smaller than $i_{j+1}$ in the permutation $ws_{i_a}s_{i_{a-1}}\cdots s_{i_{j+1}}$. The next column of $T[\alpha]$ is similarly determined, starting from $ws_{i_a}\cdots s_{i_1}$, etc.

An {\bf increasing tableau} $T$ of shape $\lambda$ is a filling with strictly increasing rows and columns. (In fact, $T[\alpha]$ is an increasing tableau.) Let ${\tt row}(T)$ be the reading word of $T$, obtained by reading the entries of $T$ along rows, from right to left, and from top to bottom. Let $\min (T)$ be the smallest label in $T$. Finally, given a reduced word ${\bf a}=a_1 a_2\ldots a_m$,
let ${\tt EGLS}({\bf a})$ be the output of the
\emph{Edelman-Greene correspondence} (see Section~2.1).

The following result shows ${\mathcal E}^{\alpha}_{\lambda^1,\ldots,\lambda^k}\in {\mathbb Z}_{\geq 0}$. It is analogous
to one on Schubert polynomials \cite[Corollary~3]{BKTY} (which our proof uses).

\begin{Theorem}
\label{claim:main}
The number ${\mathcal E}^{\alpha}_{\lambda^1,\ldots,\lambda^k}$ counts sequences of increasing tableaux
$(T_1,T_2,\ldots,T_k)$ where
\begin{itemize}
\item $T_i$ is of shape $\lambda^i$;
\item $\min T_1 > 0 , \min T_2 > d_1, \min T_3 > d_2, \ldots, \min T_k> d_{k-1}$; and
\item ${\tt row}(T_1)\cdot {\tt row}(T_2)\cdots {\tt row}(T_k)$ is a reduced word of $w[\alpha]$ such that\\
${\tt EGLS}({\tt row}(T_1)\cdot {\tt row}(T_2)\cdots {\tt row}(T_k))=T[\alpha]$.
\end{itemize}
\end{Theorem}

When $d_j=j$ for all $j\geq 1$, Theorem~\ref{claim:main} specializes to an instance
of the monomial expansion formula \cite[Theorem~5(1)]{Reiner.Shimozono} for $\kappa_{\alpha}$ (restated as
Theorem~\ref{thm:RS} below). Also, when
(\ref{eqn:kappaequalsschurcond}) holds, $k=1$, $d_1=t$ and thus
Theorem~\ref{claim:main} gives (\ref{eqn:kappaequalsschur}).

\begin{Example}
\label{exa:ispos}
The (strict) descents of $\alpha=(1,3,0,2,2,1)$ are $d_1=2, d_2=5$, and
\begin{multline}\nonumber
\kappa_{1,3,0,2,2,1}=s_{3,2}(x_1,x_2)s_{2,1,1}(x_3,x_4,x_5)+s_{3,2}(x_1,x_2)s_{2,1}(x_3,x_4,x_5)s_{1}(x_6)\\ \nonumber
+s_{3,1}(x_1,x_2)s_{2,2}(x_3,x_4,x_5)s_{1}(x_6)+s_{3,1}(x_1,x_2)s_{2,2,1}(x_3,x_4,x_5).\nonumber
\end{multline}
exhibits the claimed non-negativity of Theorem~\ref{claim:main}. 

Also, $w[\alpha]=2516743$ (one line notation)
and $T[\alpha]=\tableau{1&3&4\\2&5\\4&6\\5\\6}$.
Thus,
${\mathcal E}_{(3,2),(2,1,1),\emptyset}^{(1,3,0,2,2,1)}=
{\mathcal E}_{(3,2),(2,1),(1)}^{(1,3,0,2,2,1)}=
{\mathcal E}_{(3,1),(2,2),(1)}^{(1,3,0,2,2,1)}=
{\mathcal E}_{(3,1),(2,2,1),\emptyset}^{(1,3,0,2,2,1)}=1$ are respectively witnessed by
\[\left(\tableau{1&3&4\\2&5},\tableau{4&6\\5\\6},\emptyset\right),
\ \left(\tableau{1&3&4\\2&5},\tableau{4&6\\5},\tableau{6}\right), \ 
\left(\tableau{1&3&4\\2},\tableau{4&5\\5&6},\tableau{6}\right), \mbox{ \ and \ }
\left(\tableau{1&3&4\\2},\tableau{4&5\\5&6\\6},\emptyset\right).\]
For example, for the leftmost sequence, ${\tt EGLS}(43152\cdot 6456 \cdot \emptyset)=T[\alpha]$ holds.
\qed
\end{Example}

\subsection{The $\Omega$ polynomials}
A.~Lascoux \cite{Lascoux:trans} defines $\Omega_{\alpha}$
for $\alpha=(\alpha_1,\alpha_2,\ldots) \in {\mathbb Z}_{\geq 0}^{\infty}$ by replacing $\pi_i$ in the definition
of the key polynomials with the operator defined by
\[{\widetilde \pi}_i(f)=\partial_i(x_i(1-x_{i+1})f).\]
The initial condition is
$\Omega_{\alpha}=x_1^{\alpha_1} x_2^{\alpha_2} x_{3}^{\alpha_3}\cdots 
(=\kappa_{\alpha})$, if $\alpha$ is weakly decreasing.

The {\bf skyline diagram}
is ${\tt Skyline}(\alpha)=\{(i,y):1\leq y\leq \alpha_i\}\subset {\mathbb N}^2$.
Graphically, it is a collection of columns $\alpha_i$ high.
For instance,
\[{\tt Skyline}(1,3,0,2,2,1)=\left(\begin{matrix}
\!.\! & \!+\! & \!.\! & \!.\! &\!.\! & \!.\! \\
\!.\! & \!+\! & \!.\! & \!+\! &\!+\! & \!.\! \\
\!+\! & \!+\! & \!.\! & \!+\! &\!+\! & \!+\!
\end{matrix}\right)\]

Beginning with ${\tt Skyline}(\alpha)$, {\bf Kohnert's rule} \cite{Kohnert} 
generates
diagrams $D$ by sequentially moving any $+$
at the top of its column to the rightmost open position in its row and to its left. (The result of such a move
need not be the skyline of any $\gamma\in {\mathbb Z}_{\geq 0}^\infty$.)
Let $x^D=\prod_i x_i^{d_i}$ be the column weight where $d_i$ is the number of $+$'s in column $i$
of $D$. If the same $D$ results from a different sequence of moves, it only counts once. Kohnert's theorem states $\kappa_{\alpha}=\sum x^D$, where the
sum is over all such $D$. Extending this, we introduce:

\noindent
{\bf The $K$-Kohnert rule:} Each $+$ either moves
    as in Kohnert's rule, or stays in place \emph{and} moves. In the latter case, mark
    the original position with a ``$g$''. The $g$'s are unmovable, 
but a given $+$
    treats $g$ the same as other $+$'s when deciding if it can move,
and to where. Diagrams with
    the same occupied positions but different arrangements of $+$'s and $g$'s are counted separately.

\begin{Example}
Below, we give all $K$-Kohnert moves one step from $D$:
\[D=\left(\begin{matrix}
\!+\! & \!.\! & \!g\! & \!+\! &\!.\! \\
\!.\! & \!+\! & \!+\! & \!+\! & \!+\!
\end{matrix}\right)\mapsto
\left(\begin{matrix}
\!+\! & \!.\! & \!g\! & \!+\! & \!.\! \\
\!+\! & \!.\! & \!+\! & \!+\! & \!+\!
\end{matrix}\right), \
\left(\begin{matrix}
\!+\! & \!.\! & \!g\! & \!+\! & \!.\! \\
\!+\! & \!g\! & \!+\! & \!+\! & \!+\!
\end{matrix}\right),  \
\left(\begin{matrix}
\!+\! & \!+\! & \!g\! & \!.\! &\!.\! \\
\!.\! & \!+\! & \!+\! & \!+\! & \!+\!
\end{matrix}\right),
\]
\[ \ \ \ \ \ \ \ \ \ \ \ \ \ \ \ \ \ \ \  \ \ \ \ \ \ \ \ \ \ \ \ \ \ \ \ \  \ \ \ \ \ \ \ \
\left(\begin{matrix}
\!+\! & \!+\! & \!g\! & \!g\! &\!.\! \\
\!.\! & \!+\! & \!+\! & \!+\! & \!+\!
\end{matrix}\right), \
\left(\begin{matrix}
\!+\! & \!.\! & \!g\! & \!+\! &\!.\! \\
\!+\! & \!+\! & \!+\! & \!+\! & \!.\!
\end{matrix}\right), \
\left(\begin{matrix}
\!+\! & \!.\! & \!g\! & \!+\! &\!.\! \\
\!+\! & \!+\! & \!+\! & \!+\! & \!g\!
\end{matrix}\right).
\]
\end{Example}

Let
 \[J_{\alpha}^{(\beta)}=\sum \beta^{(\small \#\mbox{$g$'s appearing in $D$})}x^D.\]
\begin{Conjecture}
\label{conj:OmegaequalsJ}
$J_{\alpha}^{(-1)}=\Omega_{\alpha}$.
\end{Conjecture}
Conjecture~\ref{conj:OmegaequalsJ} has been checked by computer, for a wide 
range of cases up to $\alpha$ being of size $12$, leaving us convinced.
Clearly, $J_{\alpha}^{(0)}=\kappa_{\alpha}$, by Kohnert's theorem.

\begin{Example} Let $\alpha=(1,0,2)$. Then the diagrams contributing to $J_{(1,0,2)}$ are:
\[{\tt Skyline}(1,0,2)=\left(\begin{matrix}
\!.\! & \!.\! & \!+\! \\
\!+\! & \!.\! & \!+\!
\end{matrix}\right), \
\left(\begin{matrix}
\!.\! & \!+\! & \!.\! \\
\!+\! & \!.\! & \!+\!
\end{matrix}\right), \
\left(\begin{matrix}
\!+\! & \!.\! & \!.\! \\
\!+\! & \!.\! & \!+\!
\end{matrix}\right), \
\left(\begin{matrix}
\!+\! & \!.\! & \!.\! \\
\!+\! & \!+\! & \!.\!
\end{matrix}\right), \
\left(\begin{matrix}
\!.\! & \!+\! & \!.\! \\
\!+\! & \!+\! & \!.\!
\end{matrix}\right);
\]
\[\left(\begin{matrix}
\!+\! & \!g\! & \!.\! \\
\!+\! & \!.\! & \!+\!
\end{matrix}\right), \
\left(\begin{matrix}
\!+\! & \!g\! & \!.\! \\
\!+\! & \!+\! & \!.\!
\end{matrix}\right), \
\left(\begin{matrix}
\!+\! & \!.\! & \!.\! \\
\!+\! & \!+\! & \!g\!
\end{matrix}\right), \
\left(\begin{matrix}
\!.\! & \!+\! & \!.\! \\
\!+\! & \!+\! & \!g\!
\end{matrix}\right), \
\left(\begin{matrix}
\!.\! & \!+\! & \!g\! \\
\!+\! & \!.\! & \!+\!
\end{matrix}\right), \
\left(\begin{matrix}
\!+\! & \!.\! & \!g\! \\
\!+\! & \!.\! & \!+\!
\end{matrix}\right);
\left(\begin{matrix}
\!+\! & \!g\! & \!.\! \\
\!+\! & \!+\! & \!g\!
\end{matrix}\right); \
\left(\begin{matrix}
\!+\! & \!g\! & \!g\! \\
\!+\! & \!.\! & \!+\!
\end{matrix}\right).\]
Thus
\begin{multline}\nonumber
J_{(1,0,2)}=(x_1 x_3^2 + x_1 x_2 x_3 + x_1^2 x_3 +x_1^2 x_2
+x_1 x_2^2)\\ \nonumber
-(x_1^2 x_2 x_3 + x_1^2 x_2^2 +x_1^2 x_2 x_3
+x_1 x_2^2 x_3 + x_1 x_2 x_3^2
+x_1^2 x_3^2)+(x_1^2 x_2^2 x_3 + x_1^2 x_2 x_3^2).
\end{multline}

The lowest degree homogeneous
component of $\Omega_{\alpha}$ is $\kappa_{\alpha}$.  Hence any $f\in {\sf Pol}$ is a
 possibly \emph{infinite} linear combination of the $\Omega_{\alpha}$'s. Finiteness is
asserted in \cite[Chapter 5]{Poly}. We show in Section~4.2 that
the $J_{\alpha}$'s also form a (finite) basis. 

\subsection{Grothendieck polynomials}
The {\bf Grothendieck polynomial} \cite{LasSch2} is defined using the {\bf isobaric divided difference
operator} whose action on $f\in {\sf Pol}$ is given by:
\[\pi_i(f)=\partial_i((1-x_{i+1})f).\]
Declare ${\mathfrak G}_{w_0}(X)=x_1^{n-1}x_2^{n-2}\cdots x_{n-1}$ 
where $w_0$ is the long element in $S_n$. Set ${\mathfrak G}_w(X)=\pi_i({\mathfrak G}_{ws_i})$ 
if $i$ is an ascent of $w$.
The Grothendieck polynomials are known to lift $\{s_{\lambda}\}$ to  ${\sf Pol}$.

One has 
${\mathfrak G}_w={\mathfrak S}_w+\mbox{(higher degree terms)}$.
We now state the A.~Kohnert's conjecture \cite{Kohnert} for ${\mathfrak S}_w$.
The {\bf Rothe diagram} is ${\tt Rothe}(w)=\{(x,y)|y<w(x) \mbox{\ and \ } x<w^{-1}(y)\}\subset [n]\times [n]$
(indexed so that the southwest corner is labeled $(1,1)$).
Starting with ${\tt Rothe}(w)$, the Kohnert's rule generates diagrams $D$ by applying the same rules
as described for his rule for $\kappa_{\alpha}$. Then ${\mathfrak S}_w=\sum x^D$; the
sum is over all such $D$.

Analogously, we define
\[K_w^{(\beta)}=\sum_{D} \beta^{\small \mbox{($\#g$'s appearing in $D$)}} {\bf x}^D\]
where the sum is over all diagrams $D$ generated by the $K$-Kohnert rule.
For example, if $w=3142$ the diagrams contributing to $K_w^{(\beta)}$ are
\[{\tt Rothe}(3142)=\left(\begin{matrix}
\!.\!\! & \!.\! & \!.\! & \!.\!\\
\!.\!\! & \!.\! & \!.\! & \!.\!\\
\!+\!\! & \!.\! & \!+\! & \!.\!\\
\!+\!\! & \!.\! & \!.\! & \!.\!
\end{matrix}\right), \
\left(\begin{matrix}
\!\!.\!\! & \!\!.\!\! & \!.\! & \!.\!\\
\!\!.\!\! & \!\!.\!\! & \!.\! & \!.\!\\
\!\!+\!\! & \!\!+\!\! & \!.\! & \!.\!\\
\!\!+\!\! & \!\!.\!\! & \!.\! & \!.\!
\end{matrix}\right), \
\left(\begin{matrix}
\!.\!\! & \!\!.\!\! & \!.\! & \!\!.\!\\
\!.\!\! & \!\!.\!\! & \!.\! & \!\!.\!\\
\!+\!\! & \!\!+\!\! & \!g\! & \!\!.\!\\
\!+\!\! & \!\!.\!\! & \!.\! & \!\!.\!
\end{matrix}\right).\]
and hence correspondingly, $K_{3142}^{(-1)}=(x_1^2 x_3 + x_1^2 x_2) - (x_1^2 x_2 x_3)$.

        \begin{Conjecture}
          \label{conj:groth}
        $K_w^{(-1)}=\Groth_w$.
        \end{Conjecture}
Note, $K_w^{(0)}={\mathfrak S}_w$ is precisely Kohnert's conjecture.
Conjecture~\ref{conj:groth} has been checked by computer for $n\leq 7$,
and extensively for larger $n$. While Kohnert's rule for ${\mathfrak S}_w$ is 
handy, it remains mysterious, even after
\cite{Winkel}. Conjectures~\ref{conj:OmegaequalsJ} and~\ref{conj:groth} 
return to Kohnert's conjecture 
(albeit with a parameter $\beta$).

\section{Proof of Theorem~\ref{claim:main}}

\subsection{Reduced word combinatorics}
Given $w\in S_n$, let
\[{\bf a}=(a_1,a_2,\ldots a_{\ell(w)}) \mbox{\ and \ } {\bf i}=(i_1,i_2,\ldots,i_{\ell(w)}).\]
In connection to \cite{BJS}, we say the pair $({\bf a},{\bf i})$ is a {\bf stable compatible pair for $w$} if
$s_{a_1}\cdots s_{a_{\ell(w)}}$ is a reduced word for $w$ and the following two conditions
on ${\bf i}$ hold:
\begin{itemize}
\item[(cs.1)] $1\leq i_1\leq i_2\leq\cdots\leq i_{\ell(w)}<n$;
\item[(cs.2)] $a_j<a_{j+1}\implies i_j<i_{j+1}$.
\end{itemize}
We will identify $w$ with ${\bf a}$ and the associated reduced word.

The {\bf Edelman-Greene correspondence} \cite{EG} (the same basic construction is used in \cite{LasSch2})
is a bijection
\[{\tt EGLS}:({\bf a},{\bf i})\mapsto (T,U)\]
where
\begin{itemize}
\item $T$ is an increasing tableau such that ${\tt row}(T)$ is a
reduced word for ${\bf a}$; and
\item $U$ is a semistandard tableau whose multiset of labels is precisely
those in ${\bf i}$, and which has the same shape as $T$.
\end{itemize}

\noindent
{\sf {\tt EGLS} (column) insertion:} Initially insert $a_j$ into the leftmost column (of what will be $T$).
If there are no labels strictly larger than $a_j$, we place $a_j$ at the bottom of that column.
If $a_j +t$ for $t>2$ appears, we bump this $a_j+t$ to the next column to the right, replacing it with $a_j$.
The same holds if $a_j+1$ appears but not $a_j$. Finally, if both $a_j+1$ and $a_j$ already appear, we insert
$a_j+1$ into the next column to the right. Since ${\bf a}$ is assumed to be reduced, the above enumerates all possibilities. Finally at step $j$ a new box is created at a corner; in what will be $U$ we place $i_j$ .

Mildly abusing terminology, let ${\tt EGLS}({\bf a})=T$.  

\excise{\noindent
{\sf Description of $K_{-}^{0}(T)$ using Edelman-Greene:} For completeness we describe the computation of $K_{-}^0(T)$ in terms of Edelman-Greene insertion. Our description is different than, but equivalent to the definition from \cite{Reiner.Shimozono}. The first column of $K_{-}^0(T)$ is the same as the first column of $T$. Suppose
the first $j$ columns of $K_{-}^0(T)$ have been determined. To determine the $(j+1)$-st column suppose the labels in column $j+1$ of $T$ are $x_1<x_2<\ldots<x_{m}$.
Do reverse Edelman-Greene insertion of $x_{m}, x_{m-1},\ldots,x_1$ (in that order) into column $j$ of $T$. That is, we bump
the \emph{largest} label $y_m$ strictly smaller than $x_m$ (if $x_m$ already appeared in column $j$ of $T$ then $x_{m}-1$ must also appear). Repeat
for $x_{m-1},\ldots,x_1$ producing $y_m>y_{m-1}>\ldots>y_1$. Now continue by inserting these labels into column $j-1$ of $T$ etc.

Since we will not actually need the details of this algorithm in our proof, we content ourselves with an example.
\begin{Example}
If $T=\tableau{1&2&3\\2&3\\4}$ then $K_{-}^0(T)=\tableau{1&1&1\\2&2\\4}$. To obtain the second column we reverse insert $3$ and $2$ into the first column of $T$
giving $2$ and $1$ respectively. To obtain the third column we reverse insert $3$ into the second column of $T$, outputting $2$. This $2$ is reverse inserted
into the first column, giving $1$.\qed}

\end{Example}
\subsection{Formulas for Schubert polynomials}
A stable compatible pair $({\bf a},{\bf i})$ is a {\bf compatible pair for $w$}
if in addition to (cs.1) and (cs.2) the following holds:
\begin{itemize}
\item[(cs.3)] $i_j\leq a_j$.
\end{itemize}
Let ${\tt Compatible}(w)$ be the set of compatible sequences for $w$.
A rule of \cite{BJS} states:
\begin{equation}
\label{eqn:BJS}
{\mathfrak S}_w(X)=\sum_{({\bf a},{\bf i})\in {\tt Compatible}(w)}{\bf x}^{\bf i}.
\end{equation}

A {\bf descent} of $w$ is an index $j$ such that $w(j)>w(j+1)$. Let ${\tt Descents}(w)$ be the set of
descents of $w$. The following is \cite[Corollary~3]{BKTY}:

\begin{Theorem}
\label{thm:BKTY}
Let $w \in S_n$ and suppose ${\tt Descents}(w)\subseteq \{d_1<d_2<\ldots<d_k\}$. Then
\begin{equation}
\label{eqn:Schubsplit}
{\mathfrak S}_w(X)=\sum_{\lambda^1,\ldots,\lambda^k}c_{\lambda^1,\ldots,\lambda^k}^w s_{\lambda^1}(X_1)\cdots
s_{\lambda^k}(X_k)
\end{equation}
where $c_{\lambda^1,\ldots,\lambda^k}^w$ counts the number of tuples of increasing tableaux $(T_1,\ldots,T_k)$
where
\begin{itemize}
\item[(i)] $T_i$ has shape $\lambda^i$;
\item[(ii)] $\min T_1 >0, \min T_2>d_1,\ldots,\min T_k>d_{k-1}$; and
\item[(iii)] ${\tt row}(T_1)\cdots {\tt row}(T_k)$ is a reduced word of $w$.
\end{itemize}
\end{Theorem}

Assume for the remainder of the proof that
\begin{equation}
\label{eqn:asspt}
{\tt Descents}(w)\subseteq \{d_1<d_2<\ldots<d_k\}.
\end{equation}
Let
\[{\tt Tuples}(w)=\{[(T_1,U_1),(T_2,U_2),\ldots,(T_k,U_k)]\}\]
where the $T_i$'s satisfy (i), (ii) and (iii) from Theorem~\ref{thm:BKTY}, and each $U_i$ is a semistandard
tableau of shape $\lambda^i$ using the labels $d_{i-1}+1,d_{i-1}+2,\ldots, d_i$ ($d_0=0$).

\subsection{``Splitting'' the ${\tt EGLS}$ correspondence}
Assuming (\ref{eqn:asspt}) we define:
\[\Phi:{\tt Compatible}(w)\to {\tt Tuples}(w).\]

\noindent
{\sf Description of $\Phi$ (using {\tt EGLS}):} Uniquely split $({\bf a},{\bf i})\in {\tt Compatible}$ as follows
\begin{equation}
\label{eqn:thesplit}
\left(({\bf a}^{(1)},{\bf i}^{(1)}), ({\bf a}^{(2)},{\bf i}^{(2)}),\cdots, ({\bf a}^{(k)},{\bf i}^{(k)})\right)
\end{equation}
where
\begin{itemize}
\item ${\bf a}={\bf a}^{(1)}\cdots {\bf a}^{(k)}$ and
${\bf i}={\bf i}^{(1)}\cdots {\bf i}^{(k)}$
(``$\cdots$'' means concatenation); and
\item the entries of $i^{(j)}$ are contained in the set
$\{d_{j-1}+1,d_{j-1}+2,\cdots,d_j\}$.
\end{itemize}
Now define
\[\Phi(({\bf a},{\bf i})):=\left({\tt EGLS}({\bf a}^{(1)},{\bf i}^{(1)}),\cdots, {\tt EGLS}({\bf a}^{(k)},{\bf i}^{(k)})\right).\]

\begin{Proposition}
\label{prop:Phi}
The map $\Phi:{\tt Compatible}(w)\to {\tt Tuples}(w)$ is well-defined and a
bijection.
\end{Proposition}
\begin{proof}
{\sf $\Phi$ is well-defined:} The condition (i) is just says 
$T_j$ and $U_j$ have the same shape, which is
true by ${\tt EGLS}$'s description.  For (ii), the splitting says each label in ${\bf i}^{(j)}$ is strictly bigger than
$d_{j-1}$. Now by (cs.3), each label in ${\bf a}^{(j)}$ is strictly bigger than
$d_{j-1}$ as well. By ${\tt EGLS}$'s definition, the set of
labels appearing in $T_j$ is the same as that of ${\bf a}^{(j)}$;
hence (ii) holds. Lastly, ${\tt row}(T_j)$ is a reduced word for $a^{(j)}$. Then (iii)
is clear.

\noindent
{\sf $\Phi$ is a bijection:} Since ${\tt EGLS}$ is a bijective correspondence, clearly $\Phi$ is an injection. 
Consider the weight function on ${\tt Compatible}(w)$ that assigns
$({\bf a},{\bf i})$ weight ${\bf x}^{{\bf i}}$
and assigns $[(T_1,U_1),\ldots,(T_k,U_k)]$ the weight
${\bf x}^{U_1}\cdots {\bf x}^{U_k}$, where ${\bf x}^{U_i}$ is the usual monomial associated to the tableau $U_i$. Then clearly $\Phi$ is a weight-preserving
map (since ${\tt EGLS}$ is similarly weight-preserving). Hence the surjectivity of $\Phi$ 
holds by (\ref{eqn:BJS}) and Theorem~\ref{thm:BKTY}. \end{proof}

See \cite[Section~5]{Lenart} for a proof of Theorem~\ref{thm:BKTY} 
which is close to the study of the split {\tt EGLS} correspondence 
(the argument constructs certain crystal operators).

\subsection{The tableau $T[\alpha]$}
Recall $w[\alpha]\in S_{\infty}$ satisfies ${\tt code}(w[\alpha])=\alpha$.
Let $\prec$ be the pure reverse lexicographic total ordering on monomials. The Schubert polynomial ${\mathfrak S}_{w[\alpha]}$ has leading term ${\bf x}^{\alpha}$ (with respect to $\prec$).
The same is true of $\kappa_{\alpha}$ (see \cite[Corollary 7]{Reiner.Shimozono}) so
\begin{equation}
\label{eqn:lincomb}
{\mathfrak S}_{w[\alpha]}=\kappa_{\alpha}+\mbox{linear combination of other key polynomials.}
\end{equation}

Given an increasing tableau $U$, the {\bf nil left key} $K_{-}^0(U)$
is defined by \cite{LS2} (cf. \cite[p.111--114]{Reiner.Shimozono}). 
Let ${\tt sort}(\alpha)$ be the partition obtained by rearranging $\alpha$ into weakly decreasing order.
Also let ${\tt content}(T)$ the usual content vector of a semistandard tableau $T$.
This is a result of A.~Lascoux-M.-P.~Sch\"{u}tzenberger
(cf. \cite[Theorem 4]{Reiner.Shimozono}):

\begin{Theorem}
\label{thm:LS}
\[{\mathfrak S}_{w}(X)=\sum {\kappa_{{\tt content}(K_{-}^0(U))}}\]
where the sum is over all increasing tableaux $U$ of shape ${\tt sort}(\alpha)$ 
with ${\tt row}(U)=w$.
\end{Theorem}

Thus, by (\ref{eqn:lincomb}) combined with Theorem~\ref{thm:LS} there exists a 
unique increasing tableau $U[\alpha]$ 
of shape ${\tt sort}(\alpha)$ with ${\tt row}(U[\alpha])=w[\alpha]$ and 
such that $\alpha={\tt content}(K_{-}^0(U[\alpha]))$.

Let 
$F_w=\lim_{k\to \infty}{\mathfrak S}_{1^k\times w}$
be the {\bf stable Schubert polynomial} associated to $w$. This is a 
symmetric polynomial in infinitely many variables. So therefore one
has an expansion
\begin{equation}
\label{eqn:stanley.expansion}
F_w=\sum_{\lambda} a_{w,\lambda}s_{\lambda},
\end{equation}
where the $a_{w,\lambda}\in {\mathbb Z}_{\geq 0}$ are counted by increasing tableaux $A$ of shape $\lambda$ with ${\tt row}(A)=w$. 

In \cite[Theorem~4.1]{Stanley}, it is shown $a_{w,\mu(w)'}=1$
for a certain explicitly described ``maximal'' $\mu'(w)$.  Moreover
a simple description of the witnessing tableau $A[\alpha]$ is given.
Straightforwardly, $\mu'(w[\alpha])={\tt sort}(\alpha)$.
Then $T[\alpha]$ is precisely the witnessing tableau $A[\alpha]$ 
for $a_{w[\alpha],\lambda(w[\alpha])}$  (after accounting for
the fact that \cite{Stanley}'s conventions use $F_{w[\alpha]}$ for what we call $F_{w[\alpha]^{-1}}$). We leave the details to the
reader.

Finally, the expansion of Theorem~\ref{thm:LS} refines (\ref{eqn:stanley.expansion}); see, e.g., 
\cite{Reiner.Shimozono}.  Hence, 
$T[\alpha]=A[\alpha]=U[\alpha]$.  So, $T[\alpha]$ is an increasing tableau of 
shape ${\tt sort}[\alpha]$ with
${\tt row}(T[\alpha])=w[\alpha]$ and ${\tt content}(K_{-}(T[\alpha]))=\alpha$.

\subsection{Conclusion of the proof of Theorem~\ref{claim:main}:}
From the definition of ${\tt Rothe}(w[\alpha])$:
\begin{Lemma}
The descents
of $w[\alpha]$ are contained in the set of descents $d_1<d_2<\ldots<d_k$ of $\alpha$.
\end{Lemma}

 Thus,
\begin{equation}
\label{eqn:twoequals}
{\mathfrak S}_{w[\alpha]}(X)=\sum_{({\bf a},{\bf i})} {\bf x}^{\bf i}=\sum_{\lambda^1,\ldots,\lambda^k}c_{\lambda^1,\ldots,\lambda^k}^{w[\alpha]} s_{\lambda^1}(X_1)\cdots
s_{\lambda^k}(X_k).
\end{equation}
\excise{We have been careful about our choice of $T$ precisely because of (\ref{eqn:twoequals}). For example, we have ${\mathfrak S}_{2143}=\kappa_{2,0,0,0}+\kappa_{1,0,1,0}$,
and the descents of $w=2143$ are not contained in the strict descents of $\alpha=(2,0,0,0)$.}

We recall a formula \cite[Theorem~5]{Reiner.Shimozono}:

\begin{Theorem}
\label{thm:RS}
Fix an increasing tableau $T$ with ${\tt content}(K_{-}^{0}(T))=\alpha$. Then
\[\kappa_{\alpha}=\sum_{({\bf a},{\bf i})}{\bf x}^{\bf i}\]
where the sum is over compatible sequences $({\bf a}, {\bf i})$
satisfying (cs.1), (cs.2), (cs.3) and ${\tt EGLS}({\bf a})=T$.
\end{Theorem}

Two reduced words ${\bf a}$ and ${\bf a}'$ for the same permutation are in the same {\bf Coxeter-Knuth class} if
${\tt EGLS}({\bf a})={\tt EGLS}({\bf a}')=T$. 
This $T$ {\bf represents} the class. This equivalence relation
$\sim$  on reduced words is
defined by the symmetric and transitive closure of the relations:
\begin{eqnarray}
\label{eqn:therelns}
{\bf A}i(i+1)i{\bf B} & \sim  & {\bf A}(i+1)i(i+1){\bf B}\\ \nonumber
{\bf A}acb{\bf B} & \sim  & {\bf A}cab{\bf B}\\ \nonumber
{\bf A}bac{\bf B} & \sim & {\bf A}bca{\bf B} \nonumber
\end{eqnarray}
where $a<b<c$. In particular, it is true that ${\bf a} \sim {\tt row}({\tt EGLS}({\bf a}))$.

Restrict $\Phi$ to those $({\bf a},{\bf i})\in {\tt Compatible}(w[\alpha])$ such that
${\tt EGLS}({\bf a})=T[\alpha]$. Consider
$\Phi({\bf a},{\bf i})=[(T_1,U_1),\ldots,(T_k,U_k)]$. Since ${\tt EGLS}({\bf a}^{(i)})\sim {\tt row}(T_i)$, by (\ref{eqn:therelns}) we
see
\begin{equation}
\label{eqn:wesee}
{\tt row}(T_1)\cdots {\tt row}(T_k) \sim {\bf a}^{(1)}\cdots
{\bf a}^{(k)}={\bf a}.
\end{equation}
However, since we have assumed ${\tt EGLS}({\bf a})=T[\alpha]$, therefore:
\begin{equation}
\label{eqn:Treps}
{\tt EGLS}({\tt row}(T_1)\cdots {\tt row}(T_k))=T[\alpha],
\end{equation}
The other two requirements on $(T_1,\ldots,T_k)$
hold since $\Phi$ is well-defined.

Conversely, suppose $[(T_1,U_1),\ldots,(T_k,U_k)]$
has $(T_1,\ldots,T_k)$ satisfying Theorem~\ref{claim:main}'s conditions.
Since
$\Phi$ is a bijection,
$\Phi^{-1}([(T_1,U_1),\ldots,(T_k,U_k)])=({\bf a},{\bf i})\in {\tt Compatible}(w[\alpha])$.
Also, by (\ref{eqn:wesee}), ${\bf a}\sim {\tt row}(T_1)\cdots {\tt row}(T_k)$.
Now, we assumed (\ref{eqn:Treps}) holds. Hence, ${\tt EGLS}({\bf a})=T[\alpha]$ as desired. This completes the proof of the Theorem~\ref{claim:main}.\qed

\section{Additional remarks}

\subsection{Comments on Theorem~\ref{claim:main}}
Since $\kappa_{\alpha}$ specialize non-symmetric Macdonald polynomials (see, e.g., \cite[Section~5.3]{HHL}), can one extend Theorem~\ref{claim:main} in that direction?

Theorem~\ref{claim:main} implies that the key module of \cite[Section~5]{Reiner.Shimozono}
should have an action of $GL(d_1)\times GL(d_2-d_1)\times\cdots\times GL(d_k-d_{k-1})$ such that the character is
$\kappa_{\alpha}$.

V.~Reiner suggests 
a variation of Theorem~\ref{claim:main} using the plactic theory. 
The derivation should be similar, using formulas from \cite{Plactification}. 
However we are missing the analogue of 
\cite[Corollary~4]{BKTY}; cf.~\cite[Sections~7, 8]{four}.
Theorem~\ref{claim:main} naturally generalizes to Grothendieck polynomials, using \cite{BKTY:II,BKSTY};
details may appear elsewhere.

\subsection{$J_\alpha$'s form a (finite) basis of ${\sf Pol}$}
Clearly, $J_{\alpha}(X)={\bf x}^{\alpha}+\sum_{\beta\prec \alpha} c_{\beta}{\bf x}^{\beta}$.
One decomposes $f\in {\sf Pol}$ into a possibly infinite sum of $J_{\alpha}$'s:
\begin{equation}
\label{eqn:decomfin}
f=\sum_{\alpha} g_{\alpha}J_{\alpha}
\end{equation}
That is, find the $\prec$ largest
monomial ${\bf x}^{\theta_0}$ appearing in $f^{(0)}:=f$ (say with coefficient $c_{\theta_0}$) and let
$f^{(1)}:=f-c_{\theta_0}\cdot J_{\theta_t}$. Thus $f^{(1)}$ only contains monomials strictly smaller in the $\prec$ ordering.
Now repeat, defining $f^{(t+1)}:=f^{(t)}-c_{\theta_t}J_{\theta_t}$ where ${\bf x}^{\theta_t}$ is the $\prec$-largest
monomial appearing in $f^{(t)}$ etc.
Since $J_{\alpha}$ is not homogeneous, each step $t$ potentially introduces $\prec$-smaller monomials but of higher degree. However, we claim:
\begin{Proposition}
The expansion (\ref{eqn:decomfin}) is finite.
\end{Proposition}

\begin{proof}
By the $K$-Kohnert rule, 
each $\beta$ that appears in $J_{\alpha}$ is contained in the smallest
rectangle $R$ that
contains $\alpha$. So the above procedure only involves the finitely many 
diagrams contained in $R$ for one of the finitely many
initial $\alpha\in{\mathbb Z}_{\geq 0}^{\infty}$ such that ${\bf x}^{\alpha}$ is in $f$. \end{proof}

\subsection{More on the interplay of Grothendieck and the $\Omega$ polynomials}

M.~Shimozono has suggested that the expansion of ${\mathfrak G}_w$ into $\Omega_{\alpha}$
should alternate in sign, by degree. 
An explicit rule exhibiting this has been conjectured by V.~Reiner and the second author.

\section*{Acknowledgements}
AY thanks Jim Haglund, Alain Lascoux, Mark Shimozono
and Vic Reiner for inspiring discussions and correspondence. AY also thanks Oliver Pechenik and Luis Serrano for helpful comments. 
This project was initiated during a summer undergraduate research experience at UIUC supported by NSF grant
DMS 0901331. AY also was supported by NSF grant DMS 1201595 and a Helen Corley Petit endowment at UIUC.

\end{document}